\documentclass[reqno,UKenglish, 12pt]{amsart}
\usepackage[margin=3cm]{geometry}
\usepackage{amsthm} 
\usepackage{amssymb}
\usepackage{amsopn}
\usepackage{color}
\usepackage{babel} 
\usepackage{hyperref}
\usepackage{youngtab}
\usepackage{amsmath, amsfonts, amssymb} 
\usepackage{graphicx}
\usepackage{verbatim}
\usepackage{fancyvrb}
\allowdisplaybreaks 
\usepackage[utf8]{inputenc} 
\usepackage{bbold}

\usepackage{tikz}
\usetikzlibrary{topaths,calc}

\title[Sharp bounds for the largest eigenvalue]{Sharp bounds for the largest eigenvalue of the normalized hypergraph Laplace Operator}
\date{}
\author{Raffaella Mulas}
\thanks{Max Planck Institute for Mathematics in the Sciences, Leipzig, Germany}
\email{raffaella.mulas@mis.mpg.de}

\usepackage{amssymb, amsmath}
\usepackage{verbatim}
\usepackage{hyperref}

\theoremstyle{plain}
\newtheorem{theorem}{Theorem}[section]
\newtheorem{lemma}[theorem]{Lemma}
\newtheorem{corollary}[theorem]{Corollary}

\theoremstyle{definition}
\newtheorem{definition}[theorem]{Definition}

\theoremstyle{remark}
\newtheorem{remark}[theorem]{Remark}
\newtheorem{example}[theorem]{Example}

\usepackage{tikz}
\usetikzlibrary{positioning}
\tikzset{bluenode/.style={circle,fill=gray!50,minimum size=0.4cm,inner sep=0pt},}
\tikzset{rednode/.style={circle,fill=black!100,minimum size=0.4cm,inner sep=0pt},}

\begin{document}
	\maketitle
	\begin{abstract}We generalize the classical sharp bounds for the largest eigenvalue of the normalized Laplace operator, $\frac{N}{N-1}\leq \lambda_N\leq 2$, to the case of chemical hypergraphs.
	\end{abstract}
	\subsection*{Keywords}
Normalized Laplace operator, Spectral theory, Hypergraphs
	\section{Introduction}In \cite{Hypergraphs}, the author together with J\"urgen Jost introduced the notion of \emph{chemical hypergraph}, that is, a hypergraph with the additional structure that each vertex $v$ in a hyperedge $h$ is either an input, an output or both (in which case we say that $v$ is a \emph{catalyst} for $h$). They also defined, on such hypergraphs, a normalized Laplace operator that generalizes the one introduced by Chung for graphs \cite{Chung} and they investigated some properties of its spectrum. Furthermore, in a recent work \cite{Master-Stability}, the author together with Christian Kuehn and J\"urgen Jost proposed an application of this theory to the study of dynamical systems on hypergraphs.\newline
	Here we bring forward the study of the spectral properties of the hypergraph Laplacian. Particularly, we focus on the largest eigenvalue and we generalize the classical sharp bounds that are well known for graphs. As Chung showed in \cite{Chung}, given a connected graph $\Gamma$ on $N$ nodes, its largest eigenvalue $\lambda_N$ is such that
		\begin{equation}\label{eq:graphupper}
		\lambda_N\leq 2,
	\end{equation}
with equality if and only if $\Gamma$ is bipartite, and
		\begin{equation}\label{eq:graphlower}
		\lambda_N\geq \frac{N}{N-1},
	\end{equation}with equality if and only if $\Gamma$ is complete. Therefore, we can say that $2-\lambda_N$ estimates how different the graph is from being bipartite, while $\lambda_N- \frac{N}{N-1}$ quantifies how different it is from being complete. Here we generalize \eqref{eq:graphupper} and \eqref{eq:graphlower} to the case of hypergraphs.

	\subsection*{Structure of the paper}In Section \ref{section:basic} we recall some definitions from \cite{Hypergraphs} and we fix some new notation and terminology. In Section \ref{Section:Main} we state our main theorem and we prove it in Section \ref{section:proof}. Finally, in Section \ref{Section:Cheeger-like}, we discuss a corollary of our main theorem, namely, we can generalize the Cheeger-like constant $Q$ introduced in \cite{Cheeger-like-graphs} for the largest eigenvalue of graphs and prove that the lower bound $Q\leq \lambda_N$ still holds also for hypergraphs.

	\section{Basic definitions and assumptions}\label{section:basic}
	Before stating our main results, we recall some basic definitions from \cite{Hypergraphs} and we give a few new definitions that shall be useful for our discussion.
		\begin{definition}[\cite{Hypergraphs}]
				A \textbf{chemical hypergraph} is a pair $\Gamma=(V,H)$ such that $V$ is a finite set of vertices and $H$ is a set such that every element $h$ in $H$ is a pair of elements $(V_h,W_h)$ (input and output, not necessarily disjoint) in $\mathcal{P}(V)\setminus\{\emptyset\}$. The elements of $H$ are called the \textbf{oriented hyperedges}. Changing the orientation of a hyperedge $h$ means exchanging its input and output, leading to the pair $(W_h,V_h)$. 
			\end{definition}
			\begin{definition}[\cite{Hypergraphs}]
					A \textbf{catalyst} in a hyperedge $h$ is a vertex that is both an input and an output for $h$.
				\end{definition}
	
\begin{definition}[\cite{Hypergraphs}]
We say that a hypergraph $\Gamma=(V,H)$ is \textbf{connected} if, for every pair of vertices $v,w\in V$, there exists a path that connects $v$ and $w$, i.e. there exist $v_1,\dots,v_m\in V$ and $h_1,\dots,h_{m-1}\in H$ such that:
					\begin{itemize}
						\item $v_1=v$;
						\item $v_m=w$;
						\item $\{v_i,v_{i+1}\}\subseteq h_i$ for each $i=1,\dots,m-1$.
					\end{itemize}
				\end{definition}
We fix, from now on, a connected\footnote{As for the case of graphs, it is clear by the definition of the normalized Laplacian that the spectrum of a hypergraph is given by the union of the spectra of its connected components. Therefore, without loss of generality we can choose to work on connected hypergraphs.} (chemical) hypergraph $\Gamma=(V,H)$ on $N$ vertices and $M$ hyperedges. We define the \textbf{degree} of a vertex $v$ as
	\begin{equation*}
	    \deg v:=\bigl|\text{ hyperedges containing $v$ only as an input or only as an output }\bigr|
	\end{equation*}and we define the \textbf{cardinality} of a hyperedge $h$ as
	\begin{equation*}
	    |h|:=\bigl|\text{ vertices in $h$ that are either only an input or only an output }\bigr|.
	\end{equation*}Note that, in \cite{Hypergraphs}, the degree of a vertex is defined as the total number of hyperedges containing it (also as a catalyst). Here we consider this alternative definition of degree because it is more convenient in order to state our main results. Note that both definitions coincide with the usual notion of degree when we restrict to the graph case. \newline 
	
	We assume that $\deg v>0$ for each $v\in V$.
	\begin{definition}[\cite{Hypergraphs}]The \textbf{normalized Laplacian} associated to $\Gamma$ is the operator
	\begin{equation*}
	    L:\{f:V\rightarrow\mathbb{R}\}\rightarrow\{f:V\rightarrow\mathbb{R}\}
	\end{equation*}such that, given $f:V\rightarrow\mathbb{R}$ and given $v\in V$,
\begin{align*}
Lf(v):=&\frac{\sum_{h_{\text{in}}: v\text{ input}}\biggl(\sum_{v' \text{ input of }h_{\text{in}}}f(v')-\sum_{w' \text{ output of }h_{\text{in}}}f(w')\biggr)}{\deg v}+\\
&-\frac{\sum_{h_{\text{out}}: v\text{ output}}\biggl(\sum_{\hat{v} \text{ input of }h_{\text{out}}}f(\hat{v})-\sum_{\hat{w} \text{ output of }h_{\text{out}}}f(\hat{w})\biggr)}{\deg v}.
\end{align*}
	\end{definition}
We recall that $L$ has $N$ real, non-negative eigenvalues that we denote by
\begin{equation*}
    \lambda_1\leq \ldots\leq\lambda_N.
\end{equation*}These eigenvalues are invariant under changing the orientation of any hyperedge and, in particular, the largest eigenvalue on which we shall focus here can be written as
\begin{align}
	    \lambda_N&\label{eq:lambda_N}=\max_{f:V\rightarrow\mathbb{R}}\frac{\sum_{h\in H}\left(\sum_{v\text{ input of }h}f(v)-\sum_{j\text{ output of }h}f(w)\right)^2}{\sum_{i\in V}\deg(v)f(v)^2}
	    \\&\label{eq:lambda_N2}=\max_{\gamma:H\rightarrow\mathbb{R}}\frac{\sum_{v\in V}\frac{1}{\deg v}\cdot \biggl(\sum_{h_{\text{in}}: v\text{ input}}\gamma(h_{\text{in}})-\sum_{h_{\text{out}}: v\text{ output}}\gamma(h_{\text{out}})\biggr)^2}{\sum_{h\in H}\gamma(h)^2}.
	\end{align}
	The functions $f:V\rightarrow\mathbb{R}$ realizing \eqref{eq:lambda_N} are the eigenfunctions of $L$ for $\lambda_N$. We say that the functions $\gamma:H\rightarrow\mathbb{R}$ realizing \eqref{eq:lambda_N2} are the \textbf{hyperedge-eigenfunctions}. These are the eigenfunctions of the \emph{hyperedge-Laplacian}, an operator that has the same nonzero spectrum of $L$ and therefore the same largest eigenvalue. We refer the reader to \cite{Hypergraphs} for more details.

	\begin{remark}
	Because of the definition of degree that we are adopting and by definition of $L$, it is clear that, if a vertex $v\in V$ does not belong to a hyperedge $h\in H$, then the normalized Laplacian of $\Gamma$ coincides with the normalized Laplacian defined for $\Gamma'$, a hypergraph that is given by $\Gamma$ with the additional assumption that $v$ belongs to $h$ as a catalyst. Furthermore, since we are assuming that $\deg v>0$ for each $v\in V$, we are not considering vertices that are catalysts for all hyperedges in which they are contained (and such vertices would produce the eigenvalue $0$, as shown in \cite{Hypergraphs}). Therefore, without loss of generality we can focus on hypergraphs that do not have catalysts. We formalize this in the following lemma. 
	\end{remark}

	\begin{lemma}\label{lemma:catalysts}
	 Let $\Gamma=(V,H)$ be a chemical hypergraph such that there is no vertex that is a catalyst for all hyperedges in which it is contained. Let $\hat{\Gamma}:=(V,\hat{H})$, where
	 \begin{equation*}
	     \hat{H}:=\{(V_h\setminus W_h,\, W_h\setminus V_h): h=(V_h,W_h)\in H\}.
	 \end{equation*}Then, $\Gamma$ and $\hat{\Gamma}$ are isospectral.
	\end{lemma}
	\begin{proof}Since the degree does not take into account the hyperedges for which a vertex is a catalyst, it is clear by definition of $L$ that $Lf(v)$ is invariant in $\Gamma$ and in $\hat{\Gamma}$, for all $v\in V$ and for all $f:V\rightarrow\mathbb{R}$. Therefore, in particular, the spectrum of $L$ coincides for these two hypergraphs.
 	\end{proof}
	
	In view of Lemma \ref{lemma:catalysts}, without loss of generality we can focus on \emph{oriented hypergraphs}, that is, chemical hypergraphs that do not include catalysts. Oriented hypergraphs have been introduced in \cite{ReffRusnak} by Reff and Rusnak, who also introduced the non-normalized Laplacian and the adjacency matrix for such hypergraphs. The spectral properties of these operators have been widely investigated, see for instance \cite{orientedhyp2013,orientedhyp2014,orientedhyp2016,orientedhyp2017,orientedhyp2018,orientedhyp2019,orientedhyp2019-2,orientedhyp2019-3}.\newline
	
	Throughout this paper we therefore work with a fixed connected oriented hypergraph (there are no catalysts) $\Gamma=(V,H)$ on $N$ nodes and $M$ hyperedges.

		\begin{definition}[\cite{Hypergraphs}]
	We say that a hypergraph $\Gamma$ is \textbf{bipartite} if one can decompose the vertex set as a disjoint union $V=V_1\sqcup V_2$ such that, for every hyperedge $h$ of $\Gamma$, either $h$ has all its inputs in $V_1$ and all its outputs in $V_2$, or vice versa (Figure \ref{fig:bipartiteh}).
	\end{definition}
					\begin{figure}[ht]
					\begin{center}
\begin{tikzpicture}
\node (v3) at (1,0) {};
\node (v2) at (1,1) {};
\node (v1) at (1,2) {};
\node (v6) at (5,0) {};
\node (v5) at (5,1) {};
\node (v4) at (5,2) {};

\begin{scope}[fill opacity=0.5]
\filldraw[fill=gray!70] ($(v1)+(0,0.5)$) 
to[out=180,in=180] ($(v2) + (0,-0.5)$) 
to[out=0,in=180] ($(v5) + (0,-0.5)$)
to[out=0,in=0] ($(v4) + (0,0.5)$)
to[out=180,in=0] ($(v1)+(0,0.5)$);
\filldraw[fill=white!70] ($(v2)+(0,0.5)$) 
to[out=180,in=180] ($(v3) + (0,-0.5)$) 
to[out=0,in=180] ($(v6) + (0,-0.5)$)
to[out=0,in=0] ($(v5) + (0,0.5)$)
to[out=180,in=0] ($(v2)+(0,0.5)$);
\end{scope}

\fill (v1) circle (0.05) node [right] {$v_1$} node [above] {\color{black}$+$};
\fill (v2) circle (0.05) node [right] {$v_2$} node [above] {\color{black}$+$} node [below] {\color{gray}$-$};
\fill (v3) circle (0.05) node [right] {$v_3$} node [below] {\color{gray}$-$};
\fill (v4) circle (0.05) node [left] {$v_4$} node [above] {\color{black}$-$};
\fill (v5) circle (0.05) node [left] {$v_5$} node [above] {\color{black}$-$}node [below] {\color{gray}$+$};
\fill (v6) circle (0.05) node [left] {$v_6$} node [below] {\color{gray}$+$};

\node at (0,2) {\color{black}$h_1$};
\node at (0,0) {\color{gray}$h_2$};
\end{tikzpicture}
					\end{center}
					\caption{A bipartite hypergraph with $V_1=\{v_1,v_2,v_3\}$ and $V_2=\{v_4,v_5,v_6\}$.}\label{fig:bipartiteh}
				\end{figure}
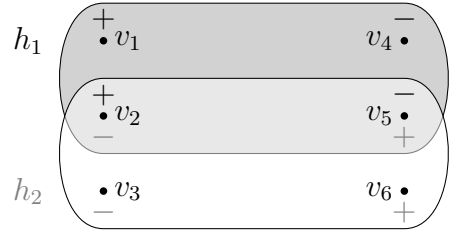
			
		\begin{definition}
	We say that a hypergraph $\hat{\Gamma}=(\hat{V},\hat{H})$ is a \textbf{sub-hypergraph} of $\Gamma=(V,H)$, denoted $\hat{\Gamma}\subset\Gamma$, if $\hat{V}\subseteq V$ and
	\begin{equation*}
	    \hat{H}=\{(V_{h}\cap \hat{V},W_{h}\cap \hat{V}):h=(V_h,W_h)\in H\}.
	\end{equation*}
	\end{definition}
	
	\begin{definition}
	Given a sub-hypergraph $\hat{\Gamma}\subset\Gamma$, we let \begin{equation*}
	     \eta(\hat{\Gamma}):=\frac{\sum_{v\in \hat{V}}\frac{\deg_{\hat{\Gamma}}(v)^2}{\deg v}}{|\hat{H}|},
	 \end{equation*}where $\deg_{\hat{\Gamma}}(v)$ denotes the degree of $v$ in $\hat{\Gamma}$ and $|\hat{H}|$ is the number of hyperedges in $\hat{\Gamma}$.
	\end{definition}We need the quantity $\eta(\hat{\Gamma})$ defined above for the statement of Theorem \ref{main-theo} below.
	
	\section{Main results}\label{Section:Main}
	We can now state our main theorem.
\begin{theorem}\label{main-theo}For every hypergraph $\Gamma$,
	 \begin{equation}\label{eq:upper}
	    \lambda_N\leq \max_{h\in H}|h|,
	 \end{equation}with equality if and only if $\Gamma$ is bipartite and $|h|$ is constant for all $h$, and
	 \begin{equation}\label{eq:lower}
	   \lambda_N\geq \max_{\hat{\Gamma}\subset\Gamma \text{ bipartite}}  \eta(\hat{\Gamma}).
	    \end{equation}
	 \end{theorem}
We prove Theorem \ref{main-theo} in Section \ref{section:proof}. Before, we discuss some consequences and examples.
	   \begin{corollary}
    For each hypergraph $\Gamma$,
        \begin{equation*}
            \lambda_N\leq N,
        \end{equation*}with equality if and only if $\Gamma$ is bipartite and each hyperedge contains all vertices.
    \end{corollary}
    \begin{remark}
	         Observe that, in the graph case, $|h|=2$ for each edge. Hence, in this case, \eqref{eq:upper} tells us that
	         \begin{equation*}
	             \lambda_N\leq 2,
	         \end{equation*}with equality if and only if the graph is bipartite. \eqref{eq:upper} is therefore a generalization of the classical upper bound for $\lambda_N$, to the case of hypergraphs.\newline
	         Also, given a graph $\Gamma$, fix a vertex $v$ and let $\hat{\Gamma}$ be the bipartite sub-graph of $\Gamma$ given by the edges that have $v$ as endpoint. Then, by \eqref{eq:lower},
	         \begin{equation*}
	             \lambda_N\geq \eta(\hat{\Gamma})=1+\sum_{w\sim v}\frac{1}{\deg w\cdot \deg v}\geq 1+\sum_{w\sim v}\frac{1}{(N-1)\cdot \deg v}=1+\frac{1}{N-1}=\frac{N}{N-1}.
	         \end{equation*}Hence, from \eqref{eq:lower}, we can re-infer the fact that $\lambda_N\geq \frac{N}{N-1}$ for graphs.
	 \end{remark}
	 \begin{example}
	 Let $\Gamma=K_N$ be the complete graph on $N$ nodes. Fix a vertex $v$ and let $\hat{\Gamma}$ be the bipartite sub-graph of $\Gamma$ given by the edges that have $v$ as endpoint. Then,
	 \begin{equation*}
	     \eta(\hat{\Gamma})=\frac{N}{N-1}=\lambda_N.
	 \end{equation*}Therefore, \eqref{eq:lower} is an equality for $K_N$.
	 \end{example}
	 \begin{example}
	 Let $\Gamma=K_N\setminus\{(v_1,v_2)\}$ be the complete graph with an edge $(v_1,v_2)$ removed. We know, from \cite{extremal}, that $\lambda_N=\frac{N+1}{N-1}$. Let $\hat{\Gamma}$ be the bipartite sub-graph of $\Gamma$ given by the edges that have either $v_1$ or $v_2$ as endpoint. Then,
	 \begin{equation*}
	     \eta(\hat{\Gamma})=\frac{N+1}{N-1}=\lambda_N.
	 \end{equation*}Therefore, \eqref{eq:lower} is an equality also in this case.
	 \end{example}
	 \begin{example}
	   For a bipartite hypergraph $\Gamma$ such that $|h|=c$ is constant for each $h$, by Theorem \ref{main-theo} $\lambda_N=c$. Also,
     \begin{equation*}
         \eta(\Gamma)=\frac{\sum_{v\in V}\deg v}{M}=\frac{\sum_{h\in H}|h|}{M}=\frac{M\cdot c}{M}=c.
     \end{equation*}Therefore, \eqref{eq:lower} is an equality.
	 \end{example}

	 \section{Proof of the main results}\label{section:proof}
	 We split the proof of Theorem \ref{main-theo} into two steps: Lemma \ref{lem:betterbound} and Lemma \ref{lemma:upperbound} below.
	 
	 \begin{lemma}\label{lem:betterbound}For every hypergraph $\Gamma$,
	 \begin{equation*}
	     \max_{\hat{\Gamma}\subset\Gamma \text{ bipartite}} \frac{\sum_{v\in \hat{V}}\frac{\deg_{\hat{\Gamma}}(v)^2}{\deg v}}{|\hat{H}|}\leq \lambda_N.
	 \end{equation*}
	 \end{lemma}
	 \begin{proof}Given a bipartite sub-hypergraph $\hat{\Gamma}\subset\Gamma$, let $\gamma':H\rightarrow\mathbb{R}$ be $1$ on $\hat{H}$ and $0$ otherwise. Then, up to changing (without loss of generality) the orientations of the hyperedges,
		\begin{align*}
		\lambda_N&=\max_{\gamma:H\rightarrow\mathbb{R}}\frac{\sum_{v\in V}\frac{1}{\deg v}\cdot \biggl(\sum_{h_{\text{in}}: v\text{ input}}\gamma(h_{\text{in}})-\sum_{h_{\text{out}}: v\text{ output}}\gamma(h_{\text{out}})\biggr)^2}{\sum_{h\in H}\gamma(h)^2}\\
					&\geq \frac{\sum_{v\in V}\frac{1}{\deg v}\cdot \biggl(\sum_{h_{\text{in}}: v\text{ input}}\gamma'(h_{\text{in}})-\sum_{h_{\text{out}}: v\text{ output}}\gamma'(h_{\text{out}})\biggr)^2}{\sum_{h\in H}\gamma'(h)^2}\\
					&\geq \frac{\sum_{v\in \hat{V}}\frac{1}{\deg v}\cdot \biggl(\sum_{h_{\text{in}}: v\text{ input}}\gamma'(h_{\text{in}})-\sum_{h_{\text{out}}: v\text{ output}}\gamma'(h_{\text{out}})\biggr)^2}{\sum_{h\in H}\gamma'(h)^2}\\
					&= \frac{\sum_{v\in \hat{V}}\frac{\deg_{\hat{\Gamma}}(v)^2}{\deg v}}{|\hat{H}|}.					\end{align*}Since the above inequality is true for all $\hat{\Gamma}$, this proves the claim.
	 \end{proof}

\begin{lemma}\label{lemma:upperbound}
        For each hypergraph $\Gamma$,
        \begin{equation*}
	    \lambda_N\leq \max_{h\in H}|h|,
	 \end{equation*}with equality if and only if $\Gamma$ is bipartite and $|h|$ is constant for all $h$.
        \end{lemma}
    \begin{proof}Let $f:V\rightarrow\mathbb{R}$ be an eigenfunction for $\lambda_N$. Then,
    \begin{align*}
        \lambda_N&=\frac{\sum_{h\in H}\left(\sum_{v\text{ input of }h}f(v)-\sum_{j\text{ output of }h}f(w)\right)^2}{\sum_{i\in V}\deg(v)f(v)^2}\\
        &\leq \frac{\sum_{h\in H}\left(\sum_{v\in h}|f(v)|\right)^2}{\sum_{i\in V}\deg(v)f(v)^2},
    \end{align*}with equality if and only if $f$ has its nonzero values on a bipartite sub-hypergraph. Now, for each $h\in H$,
    \begin{align*}
       \left( \sum_{v\in h}|f(v)|\right)^2&= \sum_{v\in h}f(v)^2+\sum_{\{v,w\}:\,v\neq w \in h}2\cdot |f(v)|\cdot|f(w)|\\
       &\leq \sum_{v\in h}f(v)^2+\sum_{\{v,w\}:\,v\neq w \in h}\biggl(f(v)^2+f(w)^2 \biggr)\\
       &=\sum_{v\in h}f(v)^2 + \sum_{v\in h}(|h|-1)f(v)^2\\
       &=|h|\cdot \sum_{v\in h}f(v)^2,
    \end{align*}with equality if and only if $|f|$ is constant on all $v\in h$. Therefore,
    \begin{align*}
        \frac{\sum_{h\in H}\left(\sum_{v\in h}|f(v)|\right)^2}{\sum_{i\in V}\deg(v)f(v)^2}&\leq \frac{\sum_{h\in H}\sum_{v\in h}|h|\cdot \sum_{v\in h}f(v)^2}{\sum_{i\in V}\deg(v)f(v)^2}\\
        &=\frac{\sum_{v\in V}\sum_{h\ni v}|h|\cdot f(v)^2}{\sum_{i\in V}\deg(v)f(v)^2}\\
        &\leq \bigl(\max_{h\in H}|h|\bigr)\cdot \frac{\sum_{v\in V}\deg(v) f(v)^2}{\sum_{i\in V}\deg(v)f(v)^2}\\
        &=\max_{h\in H}|h|,
    \end{align*}where the first inequality is an equality if and only if $|f|$ is constant (since we assuming that $\Gamma$ is connected), and the last inequality is an equality if and only if $|h|$ is constant for all $h$. Putting everything together, we have that
    \begin{equation*}
        \lambda_N\leq \max_{h\in H}|h|,
    \end{equation*}with equality if and only if $|h|$ is constant for all $|h|$ while $|f|$ is constant and it's defined on a bipartite sub-hypergraph (that is, $|f|$ is constant and $\Gamma$ is bipartite).

    \end{proof}
   
     \section{Cheeger-like constant}\label{Section:Cheeger-like}
     As a consequence of Theorem \ref{main-theo}, we can also generalize the Cheeger-like constant introduced in \cite{Cheeger-like-graphs} for the case of graphs,
     \begin{equation}\label{eq:Qgraphs}
         Q:=\max_{e=(v,w)\in E}\Biggl(\frac{1}{\deg v}+\frac{1}{\deg w}\Biggr),
     \end{equation}where $E$ is the edge set of the graph, and we can prove that the lower bound $Q\leq \lambda_N$ still holds also for hypergraphs. Furthermore, we can also show that the characterization of $Q$ proved in \cite{Cheeger-like-graphs},
     \begin{equation}\label{eq:charQgraphs}		 Q=\max_{\gamma:E\rightarrow\mathbb{R}}\frac{\sum_{v\in V}\frac{1}{\deg v}\cdot \biggl|\sum_{e_{\text{in}}: v\text{ input}}\gamma(e_{\text{in}})-\sum_{e_{\text{out}}: v\text{ output}}\gamma(e_{\text{out}})\biggr|}{\sum_{e\in E}|\gamma(e)|} \end{equation}can be extended for hypergraphs as well. Note that \eqref{eq:charQgraphs} tells us that, for graphs, we can characterize $Q$ by looking at the characterization of $\lambda_N$ in \eqref{eq:lambda_N} and then replacing the $L_2$–norm by the $L_1$–norm both in the numerator and denominator. The reason why this is interesting is that something analogous happens to the classical graph Cheeger constant $h$. It is in fact well known that, for connected graphs, $h$ bounds the first non-zero eigenvalue\footnote{In the case of graphs, the multiplicity of $0$ for the normalized Laplacian equals the number of connected components of the graph. Therefore, for connected graphs, $\lambda_2$ is the first non-zero eigenvalue. The same doesn't hold for hypergraphs, see \cite{Hypergraphs}. This is why it is not yet clear how to generalize the Cheeger constant to chemical hypergraphs.} $\lambda_2$ both above and below and it can be characterized by first looking at a characterization of $\lambda_2$ using the \emph{Rayleigh quotient} and then replacing the $L_2$–norm by the $L_1$–norm both in the numerator and denominator \cite{Chung}. Furthermore, the first Cheeger-like constant for the largest graph eigenvalue that has been introduced is the \emph{dual Cheeger constant} $\bar{h}$ \cite{dual2,dual}. What makes the two Cheeger-like constants conceptually different is the fact that $\bar{h}$ is related to the Cheeger-constant $h$ \cite{dual2} and it doesn't have a characterization analogous to the one of $Q$, in terms of the Rayleigh quotient. \newline
     
     In particular, for hypergraphs, we generalize \eqref{eq:Qgraphs} by defining
   \begin{equation*}
 Q:=\max_{h\in H}\biggl(\sum_{v\in h}\frac{1}{\deg v}\biggr).
 \end{equation*}As a direct consequence of Theorem \ref{main-theo}, we can prove the following corollary.
	 \begin{corollary}For every hypergraph $\Gamma$,
	 \begin{equation*}
			     Q\leq \lambda_N.
			\end{equation*}
	 \end{corollary}
	 \begin{proof}For each $h\in H$, let $\hat{\Gamma}_h$ be the bipartite sub-hypergraph of $\Gamma$ given only by the hyperedge
	$h$. Then,
	 \begin{equation*}
	      \eta(\hat{\Gamma}_h)=\sum_{v\in h}\frac{1}{\deg v}
	 \end{equation*}
and, by Theorem \ref{main-theo},	 \begin{equation*}
	    Q=\max_{h\in H}\eta(\hat{\Gamma}_h)\leq \lambda_N.
	 \end{equation*}
	 \end{proof}We conclude by proving that also the characterization of $Q$ in \eqref{eq:charQgraphs} can be generalized to the case of hypergraphs. In particular, the proof of Lemma \ref{lemma:Qchar} below generalizes the proof of \cite[Lemma 4]{Cheeger-like-graphs}.
	 
	  \begin{lemma}\label{lemma:Qchar}For every hypergraph,
	 \begin{equation*}
	     Q=\max_{\gamma:H\rightarrow\mathbb{R}}\frac{\sum_{v\in V}\frac{1}{\deg v}\cdot \biggl|\sum_{h_{\text{in}}: v\text{ input}}\gamma(h_{\text{in}})-\sum_{h_{\text{out}}: v\text{ output}}\gamma(h_{\text{out}})\biggr|}{\sum_{h\in H}|\gamma(h)|}.
	 \end{equation*}
	 \end{lemma}
	 \begin{proof}In order to prove that
	 \begin{equation*}
	     Q\leq \max_{\gamma:H\rightarrow\mathbb{R}}\frac{\sum_{v\in V}\frac{1}{\deg v}\cdot \biggl|\sum_{h_{\text{in}}: v\text{ input}}\gamma(h_{\text{in}})-\sum_{h_{\text{out}}: v\text{ output}}\gamma(h_{\text{out}})\biggr|}{\sum_{h\in H}|\gamma(h)|},
	 \end{equation*}fix a hyperedge $h'$ that maximizes $\sum_{v\in h}\frac{1}{\deg v}$ over all $h\in H$ and let $\gamma':H\rightarrow\mathbb{R}$ be $1$ on $h'$ and $0$ otherwise. Then, 
		\begin{align*}
		Q&=\sum_{v\in h'}\frac{1}{\deg v}\\
		&=\frac{\sum_{v\in V}\frac{1}{\deg v}\cdot \biggl|\sum_{h_{\text{in}}: v\text{ input}}\gamma'(h_{\text{in}})-\sum_{h_{\text{out}}: v\text{ output}}\gamma'(h_{\text{out}})\biggr|}{\sum_{h\in H}|\gamma'(h)|}\\
		&\leq \max_{\gamma:H\rightarrow\mathbb{R}}\frac{\sum_{v\in V}\frac{1}{\deg v}\cdot \biggl|\sum_{h_{\text{in}}: v\text{ input}}\gamma(h_{\text{in}})-\sum_{h_{\text{out}}: v\text{ output}}\gamma(h_{\text{out}})\biggr|}{\sum_{h\in H}|\gamma(h)|}.
		\end{align*}
		   
	 We now prove the inverse inequality. Let $\hat{\gamma}:H\rightarrow\mathbb{R}$ be a maximizer for
	 \begin{equation*}
	     \frac{\sum_{v\in V}\frac{1}{\deg v}\cdot \biggl|\sum_{h_{\text{in}}: v\text{ input}}\gamma(h_{\text{in}})-\sum_{h_{\text{out}}: v\text{ output}}\gamma(h_{\text{out}})\biggr|}{\sum_{h\in H}|\gamma(h)|}
	 \end{equation*}such that, without loss of generality, $\sum_{h\in H}|\hat{\gamma}(h)|=1$. Then,
	 \begin{align*}
	 Q&=\max_{h\in H}\biggl(\sum_{v\in h}\frac{1}{\deg v}\biggr)\\
	 &=\Biggl(\max_{h\in H}\biggl(\sum_{v\in h}\frac{1}{\deg v}\biggr)\Biggr)\cdot\biggl(\sum_{h\in H}|\hat{\gamma}(h)|\biggr)\\
	 &\geq \sum_{h\in H}|\hat{\gamma}(h)|\cdot\biggl(\sum_{v\in h }\frac{1}{\deg v}\biggr)\\
	 &=\sum_{v\in V}\frac{1}{\deg v}\cdot \biggl(\sum_{h\ni v}|\hat{\gamma}(h)|\biggr)\\
	 &\geq \sum_{v\in V}\frac{1}{\deg v}\cdot \biggl|\sum_{h_{\text{in}}: v\text{ input}}\hat{\gamma}(h_{\text{in}})-\sum_{h_{\text{out}}: v\text{ output}}\hat{\gamma}(h_{\text{out}})\biggr|\\
	&=\max_{\gamma:H\rightarrow\mathbb{R}}\frac{\sum_{v\in V}\frac{1}{\deg v}\cdot \biggl|\sum_{h_{\text{in}}: v\text{ input}}\gamma(h_{\text{in}})-\sum_{h_{\text{out}}: v\text{ output}}\gamma(h_{\text{out}})\biggr|}{\sum_{h\in H}|\gamma(h)|}.
					\end{align*}This proves the claim.
	 \end{proof}
    \subsection*{Acknowledgments}The author is grateful to Aida Abiad and J\"urgen Jost for the helpful comments, and to Alessandro Marfoni for giving her a roof during the COVID-19 outbreak. The results presented in this paper have been proved under that roof.
		\bibliographystyle{unsrt}
	\bibliography{SharpBounds15.05.20}	
	\end{document}